\documentclass[12pt]{amsart}
\newtheorem{theorem}{Theorem}[section]

\newtheorem{proposition}[theorem]{Proposition}
\theoremstyle{definition}
\newtheorem{definition}[theorem]{Definition}

\usepackage{enumerate}
\newtheorem{remark}[theorem]{Remark}
\usepackage[margin=0.75in]{geometry}
\usepackage{calligra}
\usepackage{xcolor, soul}
\sethlcolor{yellow}
\usepackage{mathtools}

\DeclareMathOperator{\sinc}{sinc}

\DeclareMathOperator\supp{supp}
\DeclareMathOperator\Span{span}
\DeclareMathOperator\tr{tr}

\numberwithin{equation}{section}

\def\H{{\mathbb{H}}}
\def\L{{L^2(\mathbb{H})}}
\def\N{{\mathbb{N}}}
\def\h{{\mathcal{H}}}
\def\R{{\mathbb{R}}}

\def\b{{\mathcal{B}_2}}
\def\Z{{\mathbb{Z}^n}}
\def\z{{\mathbb{Z}}}

\def\c{{\mathbb{C}}}

\def\p{{\phi}}

\newcommand\numberthis{\addtocounter{equation}{1}\tag{\theequation}}

\newcommand{\be}{\begin{equation}}
\newcommand{\ee}{\end{equation}}

\begin{document}
\title{B-splines on the Heisenberg group}

\author{S. R. Das}
\address{School of Mathematical Sciences, NISER Bhubaneswar, Jatni, Odisha 752050, India}
\email{santiranjandas100@gmail.com}
\author{P. Massopust}
\address{Center of Mathematics, Technical University of Munich, Germany}
\email{massopust@ma.tum.de}

\author{R. Radha}
\address{Department of Mathematics, Indian Institute of Technology, Madras, India}
\email{radharam@iitm.ac.in}



\begin{abstract}
The aim of this paper is to introduce a class of $B$-splines $\p_n$ on the Heisenberg group $\H$ and prove certain fundamental properties. The non-commutative group structure of $\H$ leads to very different results for $\p_n$ from those of the classical $B$-splines on the real line $\R$. 
\end{abstract}
\maketitle
\section{Introduction and background}
B-splines are an important tool in approximation and interpolation theory and have been applied to a variety of different settings ranging from numerical analysis to computer-aided design and computer graphics. The Heisenberg group initially originating in quantum mechanics is ubiquitous in mathematics. It plays an important role in such diverse areas as representation theory of nilpotent groups, the theory of partial differential equations, ergodic theory, and classical invariant theory. (See \cite{Rowe} for the omnipresence of the Heisenberg group in harmonic analysis.) The non-commutativity of the Heisenberg group provides a wide range of tools for harmonic analysis and encapsulates rich and beautiful mathematical structures. 

Recently, multiresolution structures and wavelet-type systems constructed on the Heisenberg group have been investigated (cf. for instance, \cite{aratijmpa,araticol,D,LP,Mayeli}) and B-splines have been defined on locally compact abelian groups \cite[Chapter 21.5]{C} for the purpose of constructing explicit Gabor frames for such groups. 

\subsection{Rudimentaries of B-splines}
In this paper, we consider B-splines defined in a non-abelian setting and investigate their properties. For this purpose, we recall the definition of spline and B-spline on the real line $\R$. 

Let $X := \{x_0 < x_1 < \cdots < x_k<x_{k+1}\}$ be a set of knots on the real line $\R$. A spline function, or for short a spline, of order $n$ on $[x_0,x_{k+1}]$ with knot set $X$ is a function $f: [x_0,x_{k+1}]\to\mathbb{R}$ such that
\begin{enumerate}
\item[(i)] on each interval $(x_{i-1},x_i)$, $i = 0,1,\ldots, k+1$, $f$ is a polynomial of order at most $n$;
\item[(ii)]   $f\in C^{n-2}[x_0,x_{k+1}]$. In the case $n = 1$, we define $C^{-1}[x_0,x_{k+1}]$ to be the space of piecewise constant functions on $[x_0,x_{k+1}]$.
\end{enumerate}
The function $f$ is called a {cardinal spline} if the knot set $X$ is a contiguous subset of $\Z$. As a spline of order $n$ contains $n(k+1)$ free parameters on $[x_0,x_{k+1}]$ and has to satisfy $n-1$ differentiability conditions at the $k$ interior knots $x_1, \ldots, x_k$, the set ${S}_{X,n}$ of all spline functions $f$ of order $n$ over the knot set $X$ forms a real vector space of dimension $n+k$. As the space of cardinal splines of order $n$ over a finite knot set $X$ is finite-dimensional, a convenient and powerful basis for ${S}_{X,n}$ is given by the family of polynomial cardinal B-splines introduced by Curry \& Schoenberg \cite{CS}. For a detailed study of splines and wavelets we refer to \cite{Ch} and for splines and fractal functions we refer to \cite{peter}.

We now present one (equivalent) definition of cardinal B-spline suited for our purposes and list some of their properties. It is these properties which we like to extend to the Heisenberg setting. 
\begin{definition}\label{16734}
Let $\chi_{[0,1]}$ denote the characteristic function of $[0,1]$. For $n\in \N$, set
\begin{align*}
&B_{1}:[0,1]\to [0,1], \quad x\mapsto \chi_{[0,1]}(x);\\
&B_{n} := B_{n-1}\ast B_{1}, \quad n\geq 2,\quad n\in \N.\numberthis \label{1}
\end{align*}
Then $B_n$ is called a \emph{(cardinal) polynomial B-spline of order $n$}.
\end{definition}
The adjective ``cardinal'' refers to the fact that $B_n$ is defined over the set of knots $\{0,1, \ldots, n\}$.
%
In particular, the definition implies that
\begin{enumerate}[(i)]
\item $\supp B_n = [0,n]$. 
\item $B_n > 0$ on $(0,n)$.
\item The \emph{Fourier} transform of $B_n$ is given by $\widehat{B}_n(\omega)=\left(\dfrac{1-e^{-2\pi i\omega}}{2\pi i\omega}\right)^n$.
\item Partition of Unity: $\sum\limits_{k\in \Z} B_n(x-k) = 1$, for all $x\in \R$.
\item The collection $\{B_n : n\in \N\}$ constitutes a family of piecewise polynomial functions with $B_n\in C^{n-2} [0,n]$, $n\in \N$. In the case $n=1$, we define $C^{-1}[0,1]$ as the family of piecewise constant functions on $[0,1]$.
\item $\int\limits_{-\infty}^\infty B_n(x)~dx=1$, for each $n\in\N$.
\item Distributional definition of cardinal B-spline: For any continuous function $f:\R\rightarrow\c$,
\begin{align*}
\int_{-\infty}^\infty f(x)B_n(x)~dx=\int_{[0,1]^n}f(x_1+\cdots+x_n)~dx_1\cdots dx_n.
\end{align*}
\item For $n\geq 2$, $B_n^\prime(x)=(\nabla B_{n-1})(x):=B_{n-1}(x)-B_{n-1}(x-1)$.  
\end{enumerate}
\vskip 4pt

Recently in \cite{santisecond}, twisted $B$-splines were introduced on the complex plane, based on a non-standard convolution, called a twisted convolution on $\R^2$. These twisted B-splines possess higher regularity than the classical $B$-splines on $\R^2$. The results obtained in that paper motivated us to introduce $B$-splines on the Heisenberg group and study their fundamental properties.\\

\subsection{Some remarks about the Heisenberg group}
The Heisenberg group $\mathbb{H}$ is a nilpotent {\it Lie} group whose underlying manifold is $\R\times\R\times\R$ endowed with a group operation defined by 
\[
(x,y,t)(x^\prime,y^\prime,t^\prime):=(x+x^\prime,y+y^\prime,t+t^\prime+\tfrac{1}{2}(x^\prime y-y^\prime x)),
\]
and for which a {\it Haar} measure is {\it Lebesgue} measure $dx\,dy\,dt$ on $\R^3$. By the {\it Stone--von Neumann} theorem, every infinite dimensional irreducible unitary representation on $\mathbb{H}$ is unitarily equivalent to the representation $\pi_\lambda$ given by
\begin{align*}
\pi_\lambda(x,y,t)\p(\xi)=e^{2\pi i\lambda t}e^{2\pi i\lambda(x\xi+\frac{1}{2}xy)}\p(\xi+y),
\end{align*}
for $\p\in L^2(\R)$ and $\lambda\in \R^\times:=\R\setminus\{0\}$. This representation $\pi_\lambda$ is called the {\it Schr\"{o}dinger} representation of the Heisenberg group. The group $Fourier$ transform of $f\in L^1(\H)$ is defined by
\begin{equation}\label{2}
\widehat{f}(\lambda)=\int_\H f(x,y,t)\hspace{1 mm}\pi_\lambda (x,y,t)\hspace{1 mm}dxdydt,\hspace{3 mm}\lambda\in\R^\times,
\end{equation}
where the integral is a $Bochner$ integral acting on the $Hilbert$ space $L^2(\R)$. The group {\it Fourier} transform is an isometric isomorphism between $\L$ and $L^2(\R^\times,\b;d\mu)$, where  $d\mu (\lambda)$ is the Plancherel measure $|\lambda |d\lambda$ and $\b$ is the {\it Hilbert} space of {\it Hilbert-Schmidt} operators on $L^2(\R)$ with the inner product $(T,S) :=\tr(TS^*)$. We can write \eqref{2} as
\begin{equation*}
\widehat{f}(\lambda)=\int_{\R^2}f^\lambda(x,y)\pi_\lambda(x,y,0)\ dxdy\ ,
\end{equation*}
where
\begin{equation*}
f^\lambda(x,y)=\int_{\R}f(x,y,t)e^{2\pi i\lambda t}\ dt.
\end{equation*} 
Notice that the function $f^\lambda(x,y)$ is the inverse {\it Fourier} transform of $f$ with respect to the $t$ variable. For $g\in L^1(\R^2)$, let
\begin{equation*}
W_{\lambda}(g)=\int_{\R^2}g(x,y)\pi_{\lambda}(x,y,0)\ dxdy,\ \ \text{for $\lambda\in\R^\times$.}
\end{equation*}
Using this operator, we can rewrite $\widehat{f}(\lambda)$ as $W_{\lambda}(f^{\lambda})$. When $f,g\in\L$, one can show that $f^{\lambda}, g^{\lambda}\in L^2(\R^2)$ and $W_{\lambda}$ satisfies
\begin{equation}\label{3}
\big\langle f^\lambda,g^\lambda\big\rangle_{L^2(\R^2)}=|\lambda|\big\langle W_\lambda (f^\lambda),W_\lambda (g^\lambda)\big\rangle_{\b}.
\end{equation}
In particular, when $\lambda=1$, $W_{\lambda}$ is called the {\it Weyl} transform. Moreover, $\widehat{f}(\lambda)$ is an integral operator whose kernel $K_{f^\lambda}^\lambda$ is given by
\begin{align*}
K_{f^\lambda}^\lambda(\xi,\eta)=\int_\R f^\lambda(x,\eta-\xi)e^{\pi i\lambda x(\xi+\eta)}~dx.
\end{align*}

For $f,g\in L^1(\mathbb{H})$, the group convolution of $f$ and $g$ is defined by
\begin{align}\label{hconv}
f*g(x,y,t):=\int_{\mathbb{H}}f\big((x,y,t)(u,v,s)^{-1}\big)g(u,v,s)\ du\,dv\,ds.
\end{align}

Then $\widehat{f*g}(\lambda)=\widehat{f}(\lambda)\widehat{g}(\lambda)$, $\lambda\in\R^*$. Notice that the group Fourier transform is an operator valued function and hence the right-hand side denotes the composition of two operators. Under this group convolution, $L^1(\H)$ becomes a non-commutative \emph{Banach} algebra.\\

For $(u,v,s)\in\H$, the left translation operator $L_{(u,v,s)}$ is defined by 
\[
L_{(u,v,s)}f(x,y,t) :=f((u,v,s)^{-1}(x,y,t)),\hspace{1mm}\forall\  (x,y,t)\in\H,
\]
which is a unitary operator on $\L$. Using the definitions of the left translation operator and convolution, one can show that
\begin{equation}\label{28}
L_{(u,v,s)}(f*g)=(L_{(u,v,s)}f)*g.
\end{equation}

Recall that for a locally compact group $G$, a lattice $\Gamma$ in $G$ is defined to be a discrete subgroup of $G$ which is co-compact. The standard lattice in $\H$ is taken to be $\Gamma=\{(2k,l,m):k,l,m\in\z\}$. For a study of analysis on the Heisenberg group we refer to \cite{follandphase, thangavelu}.\\

Let $0\neq\h$ denote a separable Hilbert space.
\begin{definition}
A sequence of the form $\{Ue_k:~k\in\N\}$, where $\{e_k:~k\in\N\}$ is an orthonormal basis of $\h$ and $U$ is a bounded invertible operator on $\h$, is called a Riesz basis. If $\{f_k:~k\in\N\}$ is a Riesz basis for $\overline{\Span\{f_k:~k\in\N\}}$, then it is called a Riesz sequence. 
\end{definition}
Equivalently, $\{f_k:~k\in\N\}$ is said to be a Riesz sequence if there exist constants $A,B>0$ such that
\begin{align*}
A\|\{c_{k}\}\|_{\ell^2(\N)}^2\leq\bigg\|\sum_{k\in\N}c_kf_k\bigg\|^2\leq B\|\{c_{k}\}\|_{\ell^2(\N)}^2,
\end{align*}
for all finite sequences $\{c_{k}\}$. For a detailed study of {\it Riesz} basis, we refer to \cite{C}.\\

The next result gives an important property of the translates of the classical cardinal B-splines $B_n$ on the real line $\R$.

\begin{theorem}\cite[Theorem 9.2.6]{C}\label{f}
Let $T_k$, $k\in\z$, denote the translation operator $T_k f(y)=f(y-k)$, $y\in\R$. Then, for each $n\in \N$, the sequence $\{T_kB_n\}_{k\in \z}$ is a Riesz sequence.
\end{theorem}

Note that the above result in particular implies that $\{T_kB_n\}_{k\in \z}$ is a \emph{Riesz} basis for its closed linear span in $L^2(\R)$.
\vskip 6pt

The structure of this paper is as follows. In Section 2, we introduce $B$-splines $\p_n$ on the Heisenberg group $\H$ and show that the system of left translates $\{L_{(2k,l,m)}\p_n:k,l,m\in\z\}$ possesses an upper \emph{Riesz} bound. In the next section, we prove certain integral properties of $\p_n$ such as (vi) and (vii) of Definition \ref{16734} mentioned for the $B$-splines on $\R$. Furthermore, we prove a partition-of-unity like property for $\p_n$. In Section 4, we establish that the support of $\p_n$ can be included in a box unlike specifying the exact support of classical $B$-splines $B_n$  and also evaluate the kernel of the integral operator $\widehat{\p_n}(\lambda)$, where $\widehat{\p_n}$ denotes the group Fourier transform of $\p_n$, which is operator valued unlike the classical Fourier transform. In Section 5, we look at the action of the Heisenberg vector fields on $\p_n$ and obtain the formulae in terms of finite differences of discrete Heisenberg right translations. Finally, in Section 6, we establish that unlike the classical cardinal B-splines, the $\p_n$ do not possess a symmetry property. We wish to mention that the non-commutativity of the group structure leads to complications in the computations.

\section{Definition and a note on a system of left translates of $B$-splines on $\H$}
\begin{definition}
Let $\phi_1 :=\frac{1}{\sqrt{2}}\chi_{[0,2]\times[0,1]\times[0,1]}$. Define $\phi_n= \underset{i = 1}{\overset{n}{*}}\phi_1$ where $*$ denotes the convolution as defined in \eqref{hconv}. We refer to $\phi_n$ as an $n$-th order $B$-spline on the Heisenberg group.
\end{definition}
We can write $\phi_{n+1}$ in terms of $\phi_n$ as follows:
\begin{align}\label{5}
\phi_{n+1}(x,y,t)=\frac{1}{\sqrt{2}}\int_{u=0}^2\int_{v=0}^1\int_{s=0}^1\phi_n(x-u,y-v,t-s+\tfrac{1}{2}(vx-uy))~dsdvdu,\ \ n\geq 1.
\end{align}
It is important to mention here that although it is difficult to write $\phi_2$ explicitly, we can compute $\p_2$ from $\p_2^\lambda$ in the following way:
\begin{align}\label{30}
\p_2(x,y,t)=\int_\R\int_\R\int_\R e^{-2\pi i\lambda t}e^{-\pi i\lambda(2\xi-y)(x-u)}\p_2^\lambda(u,y)|\lambda|~dud\xi d\lambda,
\end{align}
where
\begin{align}\label{6}
\p_2^\lambda(x,y)&=\begin{cases}
\frac{e^{2\pi i\lambda}\sinc^2(\lambda)}{\pi^2\lambda^2xy}\ \big(1-\cos(\pi\lambda xy)\big), & (x,y)\in (0,2]\times(0,1];\\
\frac{e^{2\pi i\lambda}\sinc^2(\lambda)}{\pi^2\lambda^2xy}\ \big(\cos(\pi\lambda x(y-1))-\cos(\pi\lambda x)\big), & (x,y)\in (0,2]\times [1,2];\\
\frac{e^{2\pi i\lambda}\sinc^2(\lambda)}{\pi^2\lambda^2xy}\ \big(\cos(\pi\lambda (x-2)y)-\cos(2\pi\lambda y)\big), &(x,y)\in[2,4]\times(0,1];\\
\frac{e^{2\pi i\lambda}\sinc^2(\lambda)}{\pi^2\lambda^2xy}\ \big(\cos(\pi\lambda(x-2y))-\cos(\pi\lambda(x+2y-xy))\big), &(x,y)\in [2,4]\times[1,2];\\
0, & \text{otherwise}.
\end{cases}
\end{align}
We shall give an outline of the proof of \eqref{30}. The well-known inversion formula for the group \emph{Fourier} transform on the Heisenberg group $\H$ is given by
\begin{align*}
f(z,t)=\int_{\R}\tr(\pi_\lambda(z,t)^*\widehat{f}(\lambda))|\lambda|\ d\lambda,
\end{align*}
for $f\in \mathcal{S}(\H)$, the space of Schwartz functions on $\H$. But
\begin{align*}
\tr(\pi_\lambda(z,t)^*\widehat{f}(\lambda))&=e^{-2\pi i\lambda t}\tr(\pi_\lambda(z)^*W_\lambda(f^\lambda))\\
&=e^{-2\pi i\lambda t}\tr(W_\lambda((T_{(-x,-y)}^t)^\lambda f^\lambda)),
\end{align*}
where $\pi_\lambda(z) :=\pi_\lambda(z,0)$ and $z :=(x,y)$. Using the fact that $W_\lambda((T_{(-x,-y)}^t)^\lambda f^\lambda))=\pi_\lambda(-x,-y)W_\lambda(f^\lambda)$, we can write down the kernel of $W_\lambda((T_{(-x,-y)}^t)^\lambda f^\lambda))$ explicitly as $e^{-\pi i\lambda(2\xi-y)x}K_{f^\lambda}^\lambda(\xi-y,\eta)$. Thus,
\begin{align*}
\tr(W_\lambda((T_{(-x,-y)}^t)^\lambda f^\lambda))&=\int_\R e^{-\pi i\lambda(2\xi-y)x}K_{f^\lambda}^\lambda(\xi-y,\xi)\ d\xi\\
&=\int_\R e^{-\pi i\lambda(2\xi-y)x}\int_\R f^\lambda(u,y)e^{\pi i\lambda u(2\xi-y)}\ dud\xi\\
&=\int_{\R}\int_{\R}e^{-\pi i\lambda(2\xi-y)(x-u)}f^\lambda(u,y)\ dud\xi.
\end{align*}
Therefore,
\begin{align*}
f(z,t)=\int_{\R}\int_{\R}\int_{\R}e^{-2\pi i\lambda t}e^{-\pi i\lambda(2\xi-y)(x-u)}f^\lambda(u,y)|\lambda|\ dud\xi d\lambda,
\end{align*}
proving \eqref{30}.\\

First, we notice that $\{L_{(2k,l,m)}\p_1:k,l,m\in\z\}$ is an orthonormal system. In fact,
\begin{align*}
\langle L_{(2k,l,m)}\p_1,L_{(2k^\prime,l^\prime,m^\prime)}\p_1\rangle_{L^2(\H)}&=\langle L_{((-2k^\prime,-l^\prime,-m^\prime)(2k,l,m))}\p_1,\p_1\rangle_{L^2(\H)}\\
&=\frac{1}{\sqrt{2}}\int_Q\p_1(x+2(k^\prime-k),y+(l^\prime-l),\\
&t+m^\prime-m+(kl^\prime-lk^\prime)+\tfrac{1}{2}(x(l^\prime-l)-2y(k^\prime-k)))~dxdydt\\
&=\frac{1}{2}\int_Q\chi_{[0,2]}(x+2(k^\prime-k))\ \chi_{[0,1]}(y+(l^\prime-l))\times\\
&\chi_{[0,1]}(t+m^\prime-m+(kl^\prime-lk^\prime)+\tfrac{1}{2}(x(l^\prime-l)-2y(k^\prime-k)))~dxdydt\\
&=\int_{2(k^\prime-k)}^{2+2(k^\prime-k)}\int_{l^\prime-l}^{1+(l^\prime-l)}\int_0^1\chi_{[0,2]}(x)\chi_{[0,1]}(y)\times\\
&\chi_{[0,1]}(t+m^\prime-m+(kl^\prime-lk^\prime)+\tfrac{1}{2}(x(l^\prime-l)-2y(k^\prime-k)))~dtdydx,
\end{align*}
by applying a change of variables. Hence, if $(k,l)\neq(k^\prime,l^\prime)$, then $\langle L_{(2k,l,m)}\p_1,L_{(2k^\prime,l^\prime,m^\prime)}\p_1\rangle_{L^2(\H)}=0$. 

For $(k,l)=(k^\prime,l^\prime)$, we get
\begin{align*}
\langle L_{(2k,l,m)}\p_1,L_{(2k^\prime,l^\prime,m^\prime)}\p_1\rangle_{L^2(\H)}&=\frac{1}{2}\int_0^2\int_0^1\int_0^1\chi_{[0,1]}(t+m^\prime-m)~dtdydx =\delta_{m,m^\prime}.
\end{align*}
Next, we shall show that $\{L_{(2k,l,m)}\p_n:k,l,m\in\z\}$ possesses an upper {\it Riesz} bound.
\begin{proposition}
For the $n$-th order $B$-spline $\p_n$ on the Heisenberg group $\exists\ B>0$ such that
\begin{align*}
\bigg\|\sum_{k,l,m\in\z}c_{k,l,m}L_{(2k,l,m)}\p_n\bigg\|_\L^2\leq B\big\|\{c_{k,l,m}\}\big\|_{\ell^2(\z^3)}^2,\ \ \ \ \forall\ \{c_{k,l,m}\}\in\ell^2(\z^3). 
\end{align*}
\end{proposition}
\begin{proof}
As $\{L_{(2k,l,m)}\p_1:k,l,m\in\z\}$ is an orthonormal system, the result is true for $n=1$. Assume that the result is true for $n=p$. In other words, $\exists\ B_1>0$ such that
\begin{align}\label{29}
\bigg\|\sum_{k,l,m\in\z}c_{k,l,m}L_{(2k,l,m)}\p_p\bigg\|_\L^2\leq B_1\big\|\{c_{k,l,m}\}\big\|_{\ell^2(\z^3)}^2,\ \ \ \ \forall\ \{c_{k,l,m}\}\in\ell^2(\z^3).
\end{align}
We know that if $f_1\in\L$ and $f_2\in L^1(\H)$, then $f_1*f_2\in\L$ and $\|f_1*f_2\|_{\L}\leq \|f_1\|_\L\|f_2\|_{L^1(\H)}$. Now, using \eqref{28} and \eqref{29}, we get
\begin{align*}
\bigg\|\sum_{k,l,m\in\z}c_{k,l,m}L_{(2k,l,m)}\p_{p+1}\bigg\|_\L^2&=\bigg\|\sum_{k,l,m\in\z}c_{k,l,m}(L_{(2k,l,m)}\p_p)*\p_1\bigg\|_\L^2\\
&\leq \bigg\|\sum_{k,l,m\in\z}c_{k,l,m}L_{(2k,l,m)}\p_p\bigg\|_\L^2\|\p_1\|_{L^1(\H)}^2\\
&\leq B_1\big\|\{c_{k,l,m}\}\big\|_{\ell^2(\z^3)}^2\|\p_1\|_{L^1(\H)}^2. 
\end{align*}
As $\|\p_1\|_{L^1(\H)}=\sqrt{2}$, we obtain
\begin{align*}
\bigg\|\sum_{k,l,m\in\z}c_{k,l,m}L_{(2k,l,m)}\p_{p+1}\bigg\|_\L^2\leq 2B_1\|\{c_{k,l,m}\}\big\|_{\ell^2(\z^3)}^2.
\end{align*}
Thus the result follows for $n=p+1$ with $B=2B_1>0$, proving the claim.
\end{proof}

\section{Some integral properties of $B$-splines on $\H$}

For $(u,v,s)\in\H$, we define the right translation operator on $\H$ by $R_{(u,v,s)}:\L\rightarrow\L$, $(R_{(u,v,s)}f)(x,y,t):=f((x,y,t)(u,v,s))$, $\forall\ (x,y,t)\in\H$, which is a unitary operator on $\L$.\\

For the rest of this paper, we denote by $Q:=[0,2]\times[0,1]\times[0,1]$ the fundamental domain of the lattice $\Gamma=\{(2k,l,m):k,l,m\in\z\}$ in $\H$.\\

Now, we prove some integral properties of $\p_n$. Although the proofs are based on induction, the non-standard nature of the third variable in the group structure makes the proofs more delicate than in the classical case. Hence, we provide the proofs.\\

First, we show that the integration of an $L^1$-function against $\p_n$ is the $n$-th iterated integral of the right translates of the function over the fundamental domain.

\begin{theorem}
For $f\in L^1(\H)$,
\begin{align*}
\int_{\R^3}f(x,y,t)\ \p_n(x,y,t)~dxdydt=\frac{1}{(\sqrt{2})^n}\int_{Q^n}&(R_{(x_2,y_2,t_2)}\cdots R_{(x_n,y_n,t_n)}f)(x_1,y_1,t_1)\cdot\\
&(dx_1\cdots dx_n)(dy_1\cdots dy_n)(dt_1\cdots dt_n),\ \ \ \ \ n\geq 2.
\end{align*}
\end{theorem}
\begin{proof}
We have $\int\limits_{\R^3}f(x,y,t)\p_1(x,y,t)~dxdydt=\frac{1}{\sqrt{2}}\int\limits_Qf(x,y,t)~dxdydt$. Using \eqref{5} for $n=1$ and then applying Fubini's theorem, we get
\begin{align*}
\int_{\R^3}f(x,y,t)\p_2(x,y,t)~dxdydt=\frac{1}{\sqrt{2}}\int_Q\bigg(\int_{\R^3}f(x,y,t)\p_1(x-u,&y-v,t-s+\tfrac{1}{2}(vx-uy))\cdot\\
&dxdydt\bigg)dudvds.
\end{align*}
Applying a change of variables, we obtain
\begin{align*}
\int_{\R^3}f(x,y,t)\p_2(x,y,t)~dxdydt=\frac{1}{\sqrt{2}}\int_Q\bigg(\int_{\R^3}f(x+u,&y+v,t+s-\tfrac{1}{2}(vx-uy))\p_1(x,y,t)\cdot\\
&dxdydt\bigg)dudvds\\
=\frac{1}{(\sqrt{2})^2}\int_{Q^2}f(x+u,y+v,t+s-\tfrac{1}{2}&(vx-uy))~dxdudydvdtds,
\end{align*}
which can be rewritten as follows:
\begin{align*}
\int_{\R^3}f(x,y,t)\p_2(x,y,t)~dxdydt=\frac{1}{(\sqrt{2})^2}\int_{Q^2}f(x_1+x_2,&y_1+y_2,t_1+t_2-\tfrac{1}{2}(y_2x_1-x_2y_1))\cdot\\
&dx_1dx_2dy_1dy_2dt_1dt_2\\
=\frac{1}{(\sqrt{2})^2}\int_{Q^2}(R_{(x_2,y_2,t_2)}f)(x_1,y_1,t_1)&~dx_1dx_2dy_1dy_2dt_1dt_2.
\end{align*}
Thus, the result is true for $n=2$. Now using the induction step for $m$, we get
\begin{align*}
\int_{\R^3}f(x,y,t)\ \p_{m+1}(x,y,t)~dxdydt&=\frac{1}{\sqrt{2}}\int_Q\bigg(\int_{\R^3}f(x,y,t)\p_m(x-x_{m+1},y-y_{m+1},\\
&t-t_{m+1}+\tfrac{1}{2}(y_{m+1}x-x_{m+1}y))~dxdydt\bigg)dx_{m+1}dy_{m+1}dt_{m+1}\\
=\frac{1}{(\sqrt{2})^{m+1}}&\int_{Q^{m+1}}(R_{(x_2,y_2,t_2)}\cdots R_{(x_{m+1},y_{m+1},t_{m+1})}f)(x_1,y_1,t_1)\cdot\\
&(dx_1\cdots dx_{m+1})(dy_1\cdots dy_{m+1})(dt_1\cdots dt_{m+1}),
\end{align*}
proving our assertion.
\end{proof}
Now we compute explicitly the integral of $\p_n$ and show that the resulting value depends on $n$.
\begin{theorem}
For $n\geq 1$,
\begin{align*}
\int_{\R^3}\p_n(x,y,t)~dxdydt=(\sqrt{2})^n.
\end{align*}
\end{theorem}
\begin{proof}
We have $\int\limits_{\R^3}\p_1(x,y,t)~dxdydt=\sqrt{2}$ and therefore the result is true for $n=1$. Observing that
\begin{align*}
\int_{\R^3}\p_{m+1}(x,y,t)~dxdydt&=\frac{1}{\sqrt{2}}\int_Q\bigg(\int_{\R^3}\p_m(x-u,y-v,t-s+\tfrac{1}{2}(vx-uy))~dxdydt\bigg)dudvds\\
&=\frac{1}{\sqrt{2}}\int_Q\bigg(\int_{\R^3}\p_m(x,y,t)~dxdydt\bigg)dudvds
\end{align*}
by \eqref{5}, and using the induction step for $m$, we get the required result.
\end{proof}
\begin{remark}
Let $f:\H\rightarrow\c$. Then, for $x,y,t\in\R$ and $k,l,m\in\z$, we can write
\begin{align*}
(L_{(2k,l,m)}f)(x,y,t)&=f(x-2k,y-l,t-m+\tfrac{1}{2}(-lx+2ky))\\
&=f((-2k,-l,0)(x,y,t-m))\\
&=(L_{(2k,l,0)}f)(x,y,t-m).\numberthis \label{10} 
\end{align*}
\end{remark}
Now we prove a partition-of-unity like property for $\p_n$.
\begin{proposition}
For $n\geq 1$ and $x,y\in\R$,
\begin{align*}
\int_0^1\sum_{k,l,m\in\z}L_{(2k,l,m)}\p_n(x,y,t)~dt=(\sqrt{2})^{n-2}.
\end{align*}
\end{proposition}
\begin{proof}
Using \eqref{10}, we get
\begin{align*}
\int_0^1\sum_{k,l,m\in\z}L_{(2k,l,m)}\p_1(x,y,t)~dt&=\sum_{k,l,m\in\z}\int_0^1L_{(2k,l,0)}\p_1(x,y,t-m)~dt\\
&=\sum_{k,l,m\in\z}\int_{-m}^{1-m}L_{(2k,l,0)}\p_1(x,y,t)~dt,
\end{align*}
by applying a change of variables. Hence,
\begin{align*}
\int_0^1\sum_{k,l,m\in\z}L_{(2k,l,m)}\p_1(x,y,t)~dt&=\sum_{k,l\in\z}\int_\R L_{(2k,l,0)}\p_1(x,y,t)~dt\\
&=\sum_{k,l\in\z}\int_\R \p_1(x-2k,y-l,t+\tfrac{1}{2}(2ky-lx))~dt\\
&=\sum_{k,l\in\z}\int_\R \p_1(x-2k,y-l,t)~dt,
\end{align*}
again by applying a change of variables. Now, using the definition of $\p_1$, we get
\begin{align*}
\int_0^1\sum_{k,l,m\in\z}L_{(2k,l,m)}\p_1(x,y,t)~dt&=\frac{1}{\sqrt{2}}\bigg(\sum_{k\in\z}\chi_{[0,2]}(x-2k)\bigg)\bigg(\sum_{l\in\z}\chi_{[0,1]}(y-l)\bigg)\int_\R \chi_{[0,1]}(t)~dt\\
&=\frac{1}{\sqrt{2}}\bigg(\sum_{k\in\z}\chi_{[2k,2k+2]}(x)\bigg)\bigg(\sum_{l\in\z}\chi_{[l,l+1]}(y)\bigg)\\
&=\frac{1}{\sqrt{2}},\ \ \ \ \forall\ x,y\in\R.
\end{align*}
Therefore, the result is true for $n=1$. Let $p\geq 1$. Then,
\begin{align*}
\int_0^1\sum_{k,l,m\in\z}L_{(2k,l,m)}\p_{p+1}(x,y,t)~dt&=\sum_{k,l\in\z}\int_\R \p_{p+1}(x-2k,y-l,t)~dt\\
&=\frac{1}{\sqrt{2}}\sum_{k,l\in\z}\int_\R\bigg(\int_Q\p_p(x-2k-u,y-l-v,\\
&t-s+\tfrac{1}{2}(v(x-2k)-u(y-l)))~dudvds\bigg)dt,
\end{align*}
by making use of \eqref{5}. Thus,
\begin{align}\label{11}
\int_0^1\sum_{k,l,m\in\z}L_{(2k,l,m)}\p_{p+1}(x,y,t)~dt&=\frac{1}{\sqrt{2}}\sum_{k,l\in\z}\int_Q\bigg(\int_\R\p_p(x-2k-u,y-l-v,t)~dt\bigg)dudvds.
\end{align}
However,
\begin{align*}
\int_{\R}L_{(2k,l,0)}\p_p(x-u,y-v,t)~dt&=\int_\R\p_p(x-u-2k,y-v-l,t+\tfrac{1}{2}(-l(x-u)+2k(y-v)))~dt\\
&=\int_\R\p_p(x-2k-u,y-l-v,t)~dt.\numberthis \label{12}
\end{align*}
Using \eqref{12} in \eqref{11}, we obtain
\begin{align*}
\int_0^1\sum_{k,l,m\in\z}L_{(2k,l,m)}\p_{p+1}(x,y,t)~dt&=\frac{1}{\sqrt{2}}\sum_{k,l\in\z}\int_Q\bigg(\int_\R L_{(2k,l,0)}\p_p(x-u,y-v,t)~dt\bigg)dudvds\\
&=\frac{1}{\sqrt{2}}\sum_{k,l\in\z}\int_Q\bigg(\sum_{m\in\z}\int_{-m}^{1-m}L_{(2k,l,0)}\p_p(x-u,y-v,t)~dt\bigg)dudvds\\
&=\frac{1}{\sqrt{2}}\int_Q\bigg(\int_0^1\sum_{k,l,m\in\z}L_{(2k,l,m)}\p_p(x-u,y-v,t)~dt\bigg)dudvds,
\end{align*}
by using \eqref{10} again. The result now follows by the induction step $p$.
\end{proof}
\section{Support and \emph{Fourier} Transform of B-Splines on $\H$}
We know that the polynomial B-spline $B_n$ has compact support for each $n$. For the B-splines $\p_n$ on $\H$, we find a box that includes the support of $\p_n$.
\begin{theorem}
We have $\supp\p_1=[0,2]\times[0,1]\times[0,1]$ and 
\[
\supp\p_n\subset[0,2n]\times[0,n]\times[-\tfrac{1}{2}(n+2)(n-1),\tfrac{1}{2}(n+1)(n+2)-2], 
\]
for $n\geq 2$.
\end{theorem}
\begin{proof}
The support of $\p_1$ can be immediatley deduced from its definition. Applying a change of variables in \eqref{5}, we get
\begin{align}\label{7}
\p_{n+1}(x,y,t)=\frac{1}{\sqrt{2}}\int_{x-2}^x\int_{y-1}^y\int_{t+\frac{1}{2}(uy-vx)-1}^{t+\frac{1}{2}(uy-vx)}\p_n(u,v,s)~dsdvdu.
\end{align}
First, we shall show that $\supp\p_n$ in the first two variables is $[0,2n]\times[0,n]$. The result is true for $n=1$. Assume that the result is true for $n=m\in\N$. Then,
\begin{align*}
\p_{m+1}(x,y,t)=\frac{1}{\sqrt{2}}\int_{[x-2,x]\cap[0,2m]}\int_{[y-1,y]\cap[0,m]}\int_{t+\frac{1}{2}(uy-vx)-1}^{t+\frac{1}{2}(uy-vx)}\p_m(u,v,s)~dsdvdu.
\end{align*}
By looking at the range of first two integrals, we see that $0<x<2(m+1)$ and $0<y<m+1$, proving that $\supp\p_{m+1}$ for the first two variables is $[0,2(m+1)]\times[0,m+1]$. Let $K_n$ be the support of $\p_n$ for the last variable and for $n\geq 1$. We aim to show that $K_n\subset[-\frac{1}{2}(n+2)(n-1),\frac{1}{2}(n+1)(n+2)-2]$, for all $n\in\N$. We will prove this again by induction on $n$. The result is trivially true when $n=1$. Let us assume that the result is true for $n=m$. Then, from \eqref{7}, we obtain
\begin{align*}
\p_{m+1}(x,y,t)=\frac{1}{\sqrt{2}}\int_{x-2}^x\int_{y-1}^y\int_{[t+\frac{1}{2}(uy-vx)-1,t+\frac{1}{2}(uy-vx)]\cap[-\frac{1}{2}(m+2)(m-1),\frac{1}{2}(m+1)(m+2)-2]}\hspace{-1 cm}\p_m(u,v,s)~dsdvdu.
\end{align*}
From the range of the last integral, one can show that $-\frac{1}{2}(m+2)(m-1)-\frac{x}{2}<t<\frac{1}{2}(m+1)(m+2)-1+y$, which in turn implies $-\frac{1}{2}(m+2)(m-1)-(m+1)<t<\frac{1}{2}(m+1)(m+2)+m$. Thus, $-\frac{1}{2}(m+1+2)(m+1-1)<t<\frac{1}{2}(m+2)(m+3)-2$. Therefore, $K_{m+1}\subset[-\frac{1}{2}(m+1+2)(m+1-1),\frac{1}{2}(m+1+1)(m+1+2)-2]$, proving the required result.  
\end{proof}
Unlike the classical case, the group Fourier transform on $\H$ is an operator valued function. So in order to find the operator value of the Fourier transform $\widehat{\p_n}(\lambda)$ of $\p_n$, we use the fact that it is an integral operator. Hence in the next result, we compute the corresponding kernel.

\begin{proposition}
Fix $\lambda\in\R^\times$. Then, the kernel of $\widehat{\p}_1(\lambda)$ is given by
\begin{align*}
K_{\p_1^\lambda}^\lambda(\xi,\eta)=\sqrt{2}e^{\pi i\lambda}\sinc(\lambda)e^{\pi i\lambda(\xi+\eta)}\sinc(\lambda(\xi+\eta))\chi_{[0,1]}(\eta-\xi).
\end{align*}
Here, we defined
\[
\sinc z := \frac{\sin\pi z}{\pi z}, \quad z\in \R.
\]
For $n\geq 2$, the kernel of $\widehat{\p}_n(\lambda)$ can be evaluated recursively as follows:
\begin{align*}
K_{\p_n^\lambda}^\lambda(\xi,\eta)=\sqrt{2}e^{\pi i\lambda}\sinc(\lambda)e^{2\pi i\lambda\eta}\int_0^1e^{-\pi i\lambda y_{n-1}}\sinc(\lambda(2\eta-y_{n-1}))K_{\p_{n-1}^\lambda}^\lambda(\xi,\eta-y_{n-1})~dy_{n-1}.
\end{align*} 
\end{proposition}
\begin{proof}
We know that $\widehat{\p}_1(\lambda)$ is an integral operator with kernel $K_{\p_1^\lambda}^\lambda$, given by
\begin{align}\label{8}
K_{\p_1^\lambda}^\lambda(\xi,\eta)=\int_\R\p_1^\lambda(x,\eta-\xi)e^{\pi i\lambda x(\xi+\eta)}~dx.
\end{align}
However,
\begin{align*}
\p_1^\lambda(x,y)&=\int_\R\p_1(x,y,t)e^{2\pi i\lambda t}~dt\\
&=\frac{1}{\sqrt{2}}\chi_{[0,2]}(x)\chi_{[0,1]}(y)\int_\R\chi_{[0,1]}(t)e^{2\pi i\lambda t}~dt\\
&=\frac{1}{\sqrt{2}}\chi_{[0,2]}(x)\chi_{[0,1]}(y)\int_0^1e^{2\pi i\lambda t}~dt\\
&=\frac{1}{2\sqrt{2}\pi i\lambda}(e^{2\pi i\lambda}-1)\chi_{[0,2]}(x)\chi_{[0,1]}(y)\\
&=\frac{1}{\sqrt{2}}e^{\pi i\lambda}\sinc(\lambda)\chi_{[0,2]}(x)\chi_{[0,1]}(y).
\end{align*}
Hence, \eqref{8} reduces to 
\begin{align*}
K_{\p_1^\lambda}^\lambda(\xi,\eta)&=\frac{1}{\sqrt{2}}e^{\pi i\lambda}\sinc(\lambda)\chi_{[0,1]}(\eta-\xi)\int_0^2e^{\pi i\lambda x(\xi+\eta)}~dx\\
&=\sqrt{2}e^{\pi i\lambda}\sinc(\lambda)e^{\pi i\lambda(\xi+\eta)}\sinc(\lambda(\xi+\eta))\chi_{[0,1]}(\eta-\xi).
\end{align*}
Let $n\geq 2$. As $\widehat{\p}_n(\lambda)=\big(\widehat{\p}_1\big)^n$, the kernel of $\widehat{\p}_n(\lambda)$ is given by
\begin{align*}
K_{\p_n^\lambda}^\lambda(\xi,\eta)&=\int_\R\int_\R\cdots\int_\R K_{\p_1^\lambda}^\lambda(\xi,y_1)K_{\p_1^\lambda}^\lambda(y_1,y_2)\cdots K_{\p_1^\lambda}^\lambda(y_{n-2},y_{n-1})K_{\p_1^\lambda}^\lambda(y_{n-1},\eta)~dy_{n-1}\cdots dy_2dy_1\\
=\big(\sqrt{2}e^{\pi i\lambda}&\sinc(\lambda)\big)^ne^{\pi i\lambda(\xi+\eta)}\int_\R\cdots\int_\R e^{2\pi i\lambda(y_1+\cdots+y_{n-1})}\sinc(\lambda(\xi+y_1))\sinc(\lambda(y_1+y_2))\times\\
&\cdots\sinc(\lambda(y_{n-2}+y_{n-1}))\sinc(\lambda(y_{n-1}+\eta))\chi_{[0,1]}(y_1-\xi)\chi_{[0,1]}(y_2-y_1)\times\\
&\cdots\chi_{[0,1]}(y_{n-1}-y_{n-2})\chi_{[0,1]}(\eta-y_{n-1})~dy_{n-1}\cdots dy_1\\
=\big(\sqrt{2}e^{\pi i\lambda}&\sinc(\lambda)\big)^ne^{\pi i\lambda(\xi+\eta)}\int_\R\bigg(\int_\R\cdots\int_\R e^{2\pi i\lambda(y_1+\cdots+y_{n-2})}\sinc(\lambda(\xi+y_1))\cdots\sinc(\lambda(y_{n-2}+y_{n-1}))\times\\
&\chi_{[0,1]}(y_1-\xi)\cdots\chi_{[0,1]}(y_{n-1}-y_{n-2})~dy_{n-2}\cdots dy_1\bigg)e^{2\pi i\lambda y_{n-1}}\times\\
&\sinc(\lambda(y_{n-1}+\eta))\chi_{[0,1]}(\eta-y_{n-1})~dy_{n-1},
\end{align*}
by an application of Fubini's theorem. But the term inside the bracket is $\frac{e^{-\pi i\lambda(\xi+y_{n-1})}}{(\sqrt{2}e^{\pi i\lambda}\sinc(\lambda))^{n-1}}K^\lambda_{\p^\lambda_{n-1}}(\xi,y_{n-1})$.
Hence, it follows that
\begin{align*}
K_{\p_n^\lambda}^\lambda(\xi,\eta)&=\sqrt{2}e^{\pi i\lambda}\sinc(\lambda)e^{\pi i\lambda(\xi+\eta)}\int_\R e^{-\pi i\lambda(\xi+y_{n-1})}K_{\p_{n-1}^\lambda}^\lambda(\xi,y_{n-1})e^{2\pi i\lambda y_{n-1}}\times\\
&\sinc(\lambda(y_{n-1}+\eta))\chi_{[0,1]}(\eta-y_{n-1})~dy_{n-1}\\
&=\sqrt{2}e^{\pi i\lambda}\sinc(\lambda)e^{\pi i\lambda\eta}\int_\R e^{\pi i\lambda y_{n-1}}\sinc(\lambda(y_{n-1}+\eta))K_{\p_{n-1}^\lambda}^\lambda(\xi,y_{n-1})\chi_{[0,1]}(\eta-y_{n-1})~dy_{n-1}.
\end{align*}
A change of variables leads to 
\begin{align*}
K_{\p_n^\lambda}^\lambda(\xi,\eta)=\sqrt{2}e^{\pi i\lambda}\sinc(\lambda)e^{2\pi i\lambda\eta}\int_0^1e^{-\pi i\lambda y_{n-1}}\sinc(\lambda(2\eta-y_{n-1}))K_{\p_{n-1}^\lambda}^\lambda(\xi,\eta-y_{n-1})~dy_{n-1}.
\end{align*}
\end{proof}
\section{Action of Vector Fields on B-Splines on $\H$}
For $(u,v,s)\in\H$, we define the finite differences of discrete Heisenberg right translations 
\begin{align*}
\nabla_1R_{(u,v,s)} &:=R_{(u,v,s)}-R_{(u-2,v,s)},\\ 
\nabla_2R_{(u,v,s)} &:=R_{(u,v,s)}-R_{(u,v-1,s)},\\
\nabla_3R_{(u,v,s)} &:=R_{(u,v,s)}-R_{(u,v,s-1)}. 
\end{align*}

We prove the iteration formulae for the action of the Heisenberg vector fields on $\p_n$ in terms of the above finite differences.
\begin{theorem}
Let $\mathbf{X}=\frac{\partial}{\partial x}-\frac{1}{2}y\frac{\partial}{\partial t}$, $\mathbf{Y}=\frac{\partial}{\partial y}+\frac{1}{2}x\frac{\partial}{\partial t}$ and $\mathbf{T}=\frac{\partial}{\partial t}$ denote the three left-invariant vector fields generating the Lie algebra $\mathfrak{h}$ of the Lie group $\H$. Then, for the $(n+1)$-th order $B$-spline on $\H$, we have the following:
\begin{align*}
(\mathbf{X}\p_{n+1})(x,y,t)=\frac{1}{\sqrt{2}}\int_0^1\int_0^1(\nabla_1R_{(0,-v,-s)})\p_n(x,y,t)~&dsdv+\frac{1}{2\sqrt{2}}\int_0^2\int_0^1(v-y)\times\\
&(\nabla_3R_{(-u,-v,0)})\p_n(x,y,t)~dvdu,\numberthis \label{13}
\end{align*}
\begin{align*}
(\mathbf{Y}\p_{n+1})(x,y,t)=\frac{1}{\sqrt{2}}\int_0^2\int_0^1(\nabla_2R_{(-u,0,-s)})\p_n(x,y,t)~&dsdu+\frac{1}{2\sqrt{2}}\int_0^2\int_0^1(x-u)\times\\
&(\nabla_3R_{(-u,-v,0)})\p_n(x,y,t)~dvdu\numberthis \label{14}
\end{align*}
and
\begin{align*}
(\mathbf{T}\p_{n+1})(x,y,t)=\frac{1}{\sqrt{2}}\int_0^2\int_0^1(\nabla_3R_{(-u,-v,0)})\p_n(x,y,t)~dvdu,\numberthis \label{15}\ \ \ \ \ \ x,y,t\in\R,\ n\geq 1.
\end{align*}
\end{theorem}
\begin{proof}
By definition of the vector fields, we have that 
\begin{gather*}
\mathbf{X}\p_{n+1}=\frac{\partial}{\partial x}\p_{n+1}-\frac{1}{2}y\frac{\partial}{\partial t}\p_{n+1},\\
\mathbf{Y}\p_{n+1}=\frac{\partial}{\partial y}\p_{n+1}+\frac{1}{2}x\frac{\partial}{\partial t}\p_{n+1},\\ 
\mathbf{T}\p_{n+1}=\frac{\partial}{\partial t}\p_{n+1}, \quad n\geq 1. 
\end{gather*}
Using \eqref{5}, we get
\begin{align*}
\frac{\partial}{\partial x}\p_{n+1}(x,y,t)&=\frac{1}{\sqrt{2}}\int_Q\frac{\partial}{\partial x}\p_n(x-u,y-v,t-s+\tfrac{1}{2}(vx-uy))~dudvds\\
&=\frac{1}{\sqrt{2}}\int_0^1\int_0^1-\bigg(\int_0^2\frac{\partial}{\partial u}\p_n(x-u,y-v,t-s+\tfrac{1}{2}(vx-uy))~du\bigg)dsdv+\\
&\frac{1}{\sqrt{2}}\int_0^2\int_0^1-\frac{v}{2}\bigg(\int_0^1\frac{\partial}{\partial s}\p_n(x-u,y-v,t-s+\tfrac{1}{2}(vx-uy))~ds\bigg)dvdu\\
=\frac{1}{\sqrt{2}}\int_0^1\int_0^1-&\big(\p_n(x-2,y-v,t-s+\tfrac{1}{2}(vx-2y))-\p_n(x,y-v,t-s+\tfrac{1}{2}vx)\big)dsdv+\\
\frac{1}{2\sqrt{2}}\int_0^2\int_0^1-v&\big(\p_n(x-u,y-v,t-1+\tfrac{1}{2}(vx-uy))-\p_n(x-u,y-v,t+\tfrac{1}{2}(vx-uy))\big)dvdu\\
=\frac{1}{\sqrt{2}}\int_0^1&\int_0^1-\big(\p_n((x,y,t)(-2,-v,-s))-\p_n((x,y,t)(0,-v,-s))\big)dsdv\ +\\
&\frac{1}{2\sqrt{2}}\int_0^2\int_0^1-v\big(\p_n((x,y,t)(-u,-v,-1))-\p_n((x,y,t)(-u,-v,0))\big)dvdu\\
=\frac{1}{\sqrt{2}}\int_0^1\int_0^1&-\big(R_{(-2,-v,-s)}\p_n(x,y,t)-R_{(0,-v,-s)}\p_n(x,y,t)\big)dsdv\ +\\
&\frac{1}{2\sqrt{2}}\int_0^2\int_0^1-v\big(R_{(-u,-v,-1)}\p_n(x,y,t)-R_{(-u,-v,0)}\p_n(x,y,t)\big)dvdu\\
=\frac{1}{\sqrt{2}}\int_0^1\int_0^1&(\nabla_1R_{(0,-v,-s)})\p_n(x,y,t)~dsdv+\frac{1}{2\sqrt{2}}\int_0^2\int_0^1v(\nabla_3R_{(-u,-v,0)})\p_n(x,y,t)~dvdu.
\end{align*}
Proceeding in a similar way, we obtain
\begin{align*}
\frac{\partial}{\partial y}\p_{n+1}(x,y,t)=\frac{1}{\sqrt{2}}\int_0^2\int_0^1&(\nabla_2R_{(-u,0,-s)})\p_n(x,y,t)~dsdu-\\
&\frac{1}{2\sqrt{2}}\int_0^2\int_0^1u(\nabla_3R_{(-u,-v,0)})\p_n(x,y,t)~dvdu\numberthis \label{16}
\end{align*}
and
\begin{align}\label{17}
\frac{\partial}{\partial t}\p_{n+1}(x,y,t)=\frac{1}{\sqrt{2}}\int_0^2\int_0^1(\nabla_3R_{(-u,-v,0)})\p_n(x,y,t)~dvdu.
\end{align}
Therefore,
\begin{align*}
(\mathbf{X}\p_{n+1})(x,y,t)=\frac{1}{\sqrt{2}}\int_0^1\int_0^1&(\nabla_1R_{(0,-v,-s)})\p_n(x,y,t)~dsdv\ +\\
\frac{1}{2\sqrt{2}}\int_0^2\int_0^1v(\nabla_3R_{(-u,-v,0)})\p_n(x,y,t)~dvdu\ & -\frac{y}{2\sqrt{2}}\int_0^2\int_0^1(\nabla_3R_{(-u,-v,0)})\p_n(x,y,t)~dvdu\\
=\frac{1}{\sqrt{2}}\int_0^1\int_0^1(\nabla_1R_{(0,-v,-s)})\p_n(x,y,t)~dsdv\ +&\frac{1}{2\sqrt{2}}\int_0^2\int_0^1(v-y)(\nabla_3R_{(-u,-v,0)})\p_n(x,y,t)~dvdu,
\end{align*}
proving \eqref{13}. Similarly, using \eqref{16} and \eqref{17} in $\mathbf{Y}$ and $\mathbf{T}$, we obtain \eqref{14} and \eqref{15}.
\end{proof}
\section{Lack of a symmetric property for $B$-splines on $\H$} 
We know that the classical $B$-spline of order $n$ posses the symmetry property
\[
T_{-\frac{n}{2}}B_n(x)=T_{-\frac{n}{2}}B_n(-x),\quad \forall\ x\in\R. 
\]
It is natural to expect that the translations in the first two variables in a B-splines on $\H$ act in a similar way. Hence, we ask whether there exists a sequence $\{\alpha_n\}_{n\in\N}$ such that
\begin{align}\label{18}
L_{(-n,-\frac{n}{2},-\alpha_n)}\p_n(x,y,t)=L_{(-n,-\frac{n}{2},-\alpha_n)}\p_n(-x,-y,-t),
\end{align}
as $(x,y,t)^{-1}=(-x,-y,-t)$, for $(x,y,t)\in\H$. This is true when $n=1$, namely when $\alpha_1=\frac{1}{2}$. However, for $n\geq 2$, we prove the following
\begin{theorem}
There does not exist any sequence $\{\alpha_n\}_{n\in\N}$ such that
\begin{align*}
L_{(-n,-\frac{n}{2},-\alpha_n)}\p_n(x,y,t)=L_{(-n,-\frac{n}{2},-\alpha_n)}\p_n(-x,-y,-t),
\end{align*}
for $n\geq 2$.
\end{theorem}
\begin{proof}
Let $n=2$. Suppose \eqref{18} were true. Then we have
\begin{align}\label{19}
L_{(-2,-1,-\alpha_2)}\p_2(x,y,t)=L_{(-2,-1,-\alpha_2)}\p_2(-x,-y,-t).
\end{align}
In other words,
\begin{align*}
\p_2(x+2,y+1,t+\alpha_2+(\tfrac{x}{2}-y))=\p_2(2-x,1-y,\alpha_2-t-(\tfrac{x}{2}-y)).
\end{align*}
Using \eqref{5}, we get
\begin{align*}
L_{(-2,-1,-\alpha_2)}\p_2(x,y,t)=&\frac{1}{\sqrt{2}}\int_QL_{(-1,-\frac{1}{2},-\alpha_1)}\p_1(x+1-u,y+\tfrac{1}{2}-v,\\
&t+\alpha_2-\alpha_1-s+\tfrac{1}{2}(v(x+1)-u(y+\tfrac{1}{2}))+\tfrac{1}{2}(\tfrac{x}{2}-y))~dudvds\\
=\frac{1}{\sqrt{2}}\int_Q&L_{(-1,-\tfrac{1}{2},-\alpha_1)}\p_1(-x-1+u,-y-\tfrac{1}{2}+v,\\
&-t-\alpha_2+\alpha_1+s-\tfrac{1}{2}(v(x+1)-u(y+\tfrac{1}{2}))-\tfrac{1}{2}(\tfrac{x}{2}-y))~dudvds,
\end{align*}
as \eqref{18} is true for $n=1$. Simplifying further, we obtain
\begin{align*}
L_{(-2,-1,-\alpha_2)}\p_2(x,y,t)=&\frac{1}{\sqrt{2}}\int_Q\p_1(-x+u,-y+v,\\
&-t-\alpha_2+2\alpha_1+s+\tfrac{1}{2}(u(y+1)-v(x+2))-(\tfrac{x}{2}-y))~dudvds\\
=\frac{1}{\sqrt{2}}&\int_Q\p_1((2,1,-\alpha_2+2\alpha_1+1)(-x-2+u,-y-1+v,\\
&-t+s-1+\tfrac{1}{2}(uy-vx)))~dudvds\\
=\frac{1}{\sqrt{2}}\int_Q&\p_1((2,1,-\alpha_2+2\alpha_1+1)(-x-u,-y-v,\\
&-t-s-\tfrac{1}{2}(uy-vx)-\tfrac{1}{2}(x-2y)))~dudvds,
\end{align*}
by applying a change of variables. Therefore,
\begin{align*}
L_{(-2,-1,-\alpha_2)}\p_2(x,y,t)=&\frac{1}{\sqrt{2}}\int_Q\p_1(-x-u+2,-y-v+1,\\
&-t-s-\alpha_2+2\alpha_1+1+\tfrac{1}{2}(v(x+2)-u(y+1))-(x-2y))~dudvds.\numberthis \label{20}
\end{align*}
On the other hand, using \eqref{5}, we get
\begin{align*}
L_{(-2,-1,-\alpha_2)}\p_2(-x,-y,-t)&=\frac{1}{\sqrt{2}}\int_Q\p_1(2-x-u,1-y-v,\\
&\alpha_2-t-(\tfrac{x}{2}-y)-s+\tfrac{1}{2}(v(2-x)-u(1-y)))~dudvds.\numberthis \label{21}
\end{align*}
We observe that the limits of the integral and the values in the first two components are the same in \eqref{20} and \eqref{21}. Suppose \eqref{19} were true. Then, the third components of \eqref{20} and \eqref{21} would also have to be the same for a.e. $ u\in[0,2],\ v\in[0,1]$. This, in turn, would lead to
\begin{align*}
x(\tfrac{1}{2}-v)+y(u-1)=2(\alpha_1-\alpha_2)=:k_0,
\end{align*}
for some constant $k_0$, which is impossible.
\end{proof}
We have observed that complications of the fundamental results of $\p_n$ is due to the non-abelian structure on $\H$ which is a lot harder than the classical abelian setting in $\R$. We conclude that it is really worth to go in for further investigations of these $B$-splines $\p_n$ although it might be time consuming.
\section*{Acknowledgement}
We thank Professor G.B. Folland for his valuable suggestions in obtaining the formula \eqref{30} in the paper. Initial part of this work was done while one of the authors (R.R) was visiting the Technical University of Munich, Germany, from October $2019$ to September $2020$. She sincerely thanks Professor Massimo Fornasier and Professor Peter Massopust, TUM, for their kind invitation and excellent hospitality during the entire period of her visit.

\bibliographystyle{amsplain}
\bibliography{Hsplines}
\end{document}